\documentclass[12pt]{article}

\usepackage{mathtools}
\usepackage{amsmath,amssymb,amsbsy}
\usepackage{graphicx}
\usepackage{multirow}

\usepackage{amsfonts}
\usepackage{amsthm}
\usepackage{amssymb}
\usepackage{graphicx, enumerate}
\usepackage{amsmath}
\usepackage{latexsym}
\usepackage{longtable}
\usepackage{tabularx}
\usepackage{amsmath}
\usepackage{amsfonts}
\usepackage{amsthm}
\usepackage{setspace}
\usepackage{graphicx}
\usepackage{float}
\usepackage{rotating}
\usepackage{tikz}
\usepackage{verbatim}

\newtheorem{thm}{Theorem}[section]

\newtheorem{obs}[thm]{Observation}

\newtheorem{cor}[thm]{Corollary}
\newtheorem{prob}[thm]{Problem}

\newtheorem{conj}[thm]{Conjecture}

\makeatletter
\def\imod#1{\allowbreak\mkern10mu({\operator@font mod}\,\,#1)}
\makeatother

\newcommand{\cQ}{\mathcal{Q}}
\newcommand{\zet}{\mathbb{Z}}

\def\gcd{\mathop{\rm gcd}\nolimits}

\newcommand{\gA}{\mathcal{A}}
\def\lcm{\mathop{\rm lcm}\nolimits}

\begin{document}
\title{Group distance magic Cartesian product of two cycles}

\author{Sylwia Cichacz$^{1}$\footnote{This work was partially supported by the Faculty of Applied Mathematics AGH UST statutory tasks within subsidy of Ministry of Science and Higher Education.}, Pawe{\l} Dyrlaga$^1$, Dalibor Froncek$^2$\\
$^1$AGH University of Science and Technology, Poland\\
$^2$University of Minnesota Duluth, U.S.A.}

\maketitle
\begin{abstract}
Let $G=(V,E)$ be a graph and $\Gamma $ an Abelian group both of order $n$. A
$\Gamma$-distance magic labeling of $G$ is a bijection $\ell \colon
V\rightarrow \Gamma $ for which there exists $\mu \in \Gamma $ such that $%
\sum_{x\in N(v)}\ell (x)=\mu $ for all $v\in V$, where $N(v)$ is the
neighborhood of $v$. Froncek 
showed that the Cartesian product $C_m \square C_n$, $m, n\geq3$ is a $\zet_{mn}$-distance magic graph if and only if  $mn$ is even. It is also known that if $mn$ is even then $C_m \square C_n$ has $\zet_{\alpha}\times \gA$-magic labeling for any $\alpha \equiv 0 \pmod {\lcm(m,n)}$ and any Abelian group $\gA$ of order $mn/\alpha$. 
However,  the full characterization of group distance magic Cartesian product of two cycles is still unknown.\\

In the paper we  make progress towards the complete solution this problem by proving some necessary conditions. We further prove that for $n$ even  the graph $C_{n}\square C_{n}$ has a $\Gamma$-distance magic labeling for  any Abelian group $\Gamma$ of order $n^{2}$. Moreover we show that if  $m\neq n$, then there does not exist a $(\zet_2)^{m+n}$-distance magic labeling of the Cartesian product   $C_{2^m} \square  C_{2^{n}}$. We also give necessary and sufficient condition for  $C_{m} \square  C_{n}$ with $\gcd(m,n)=1$  to   be $\Gamma$-distance magic.
\end{abstract}

 \section{Introduction}\label{sec:intro}
 All graphs $G=(V,E)$ are finite undirected simple graphs. For standard graph theoretic notation and definitions we refer to Diestel \cite{Diest}.

 In 1963 Sedl\'{a}\v{c}ek  \cite{ref_Sed} noticed the following connection between a~magic square $M$ of size $n\times n$ and an edge labeling of the complete bipartite graph $K_{n,n}$. Namely, assigning the entry in row $i$ and column $j$ of the magic square to the edge connecting the $i$-th vertex in one partite set to the $j$-th vertex in the other set,  we obtain the sum of labels of edges incident with  each vertex equal to the magic square constant. This type of labeling became known as \emph{magic labeling}, \emph{supermagic labeling}  or \emph{vertex-magic edge labeling}. Another concept in graph labeling that was motivated by the construction of magic squares labels vertices instead. A \emph{distance magic labeling} of a graph $G$ of order $n$ is a bijection $\ell:V\rightarrow\{1,2,\dots,n\}$ with the property that there exists a positive integer $\mu$ such that
 $$w(v)=\sum_{u\in N(v)}\ell(u)=\mu$$
 for all $v\in V$, where $N(v)$ is the open neighborhood of $v$ and  $w(v)$ is called the \emph{weight of the vertex} $v$. The constant $\mu$ is called the \emph{magic constant} of the labeling $f$. Any graph admitting a distance magic labeling is called a \emph{distance magic graph}.

We recall one out of four standard graph products (see \cite{IK}). Let $G$
and $H$ be two graphs. The \emph{Cartesian product} $G\square H$ is a  graph with vertex set $V(G)\times V(H)$. Two vertices $(g,h)$ and $(g^{\prime },h^{\prime })$ are adjacent  if and only if either $g=g^{\prime }$ and $h$ is adjacent
with $h^{\prime }$ in $H$, or $h=h^{\prime }$ and $g$ is adjacent with $
g^{\prime }$ in $G$.

Rao at al. proved the following result for Cartesian product of cycles in \cite{RSP}.
\begin{thm}[\cite{RSP}]\label{thm:cart_cycle_integers}The Cartesian product $C_m \square C_n$, $ m,n\geq3$ is a distance magic graph if and only if
$m,n \equiv 2\imod 4$ and $m=n$.
\end{thm}

Assume $\Gamma$ is a finite Abelian group of order $n$ with the operation denoted by $+$.  For convenience we will write $ka$ to denote $a + a + \ldots + a$ (where the element $a$ appears $k$ times), $-a$ to denote the inverse of $a$, and use $a - b$ instead of $a+(-b)$.  Moreover, the notation $\sum_{a\in S}{a}$ will be used as a short form for $a_1+a_2+a_3+\dots+a_t$, where $a_1, a_2, a_3, \dots,a_t$ are all elements of the set $S$. The identity element of $\Gamma$ will be denoted by $0$.
Recall that any group element $\iota\in\Gamma$ of order 2 (i.e., $\iota\neq 0$ and $2\iota=0$) is called an \emph{involution}.

The magic labeling (in the classical point of view)  with labels being the elements of an Abelian group has been studied for a long time (see papers by  Stanley~\cite{ref_Sta,ref_Sta2}). Therefore, it was a natural step to label the vertices of a graph $G$ with elements of an Abelian group also in the case of distance magic labeling. This concept was introduced by~Froncek in~\cite{Fro1}.\\

A $\Gamma$\emph{-distance magic labeling} of a graph $G = (V, E)$ with $|V| = n$ is a bijection $\ell$ from $V$ to an Abelian group $\Gamma$ of order $n$
such that the weight $w(v) =\sum_{u\in N(v)}\ell(u)$ of every vertex $v \in V$ is equal to the same element $\mu\in \Gamma$, called the \emph{magic
constant}. A graph $G$ is called a \emph{group distance magic graph} if there exists a $\Gamma$-distance magic labeling for every Abelian
group $\Gamma$ of order $|V(G)|$.\\

First result on $\Gamma$-distance magic labeling of Cartesian product of cycles was proved by~\cite{Fro1}:

\begin{thm}\label{thm:cart_cycle_cyclic}\emph{(\cite{Fro1})} The Cartesian product $C_m \square C_n$, $m, n\geq3$ is $\zet_{mn}$-distance magic graph if and only if
 $mn$ is even.
\end{thm}

The result was later improved by Cichacz~\cite{ref_CicAus}.
\begin{thm}\label{thm:cart_cycle_Sylwia}\emph{(\cite{ref_CicAus})}
Let $m$  or $n$ be even and $l =\lcm(m,n)$. Then $C_m \square C_n$ has a  $\zet_{\alpha}\times \gA$-magic labeling for any $\alpha \equiv 0 \pmod {l}$ and any Abelian group $\gA$ of order $mn/\alpha$.
\end{thm}

The following related results were also proved in the respective papers.

\begin{thm}\label{thm:C_2^n_by_Z_2}\emph{(\cite{Fro1})} 
	The graph $C_{2^n}\square C_{2^n}$ has a $(\zet_2)^{2n}$-distance magic labeling for $n \geq 2$ and $\mu = (0,0,\ldots,0)$.
\end{thm}

\begin{thm}\label{thm:odd-m,n}\emph{(\cite{ref_CicAus})}
	If $m,n$ are odd, then $C_m \square C_n$ is not $\Gamma$-distance magic graph for  any Abelian group $\Gamma$ of order $mn$.
\end{thm}

The following general problem is still widely open.

\begin{prob}\label{prob:general}\emph{(\cite{Fro1})}  
	For a given graph $C_m \square C_n$, determine all Abelian groups $\Gamma$ such that the graph $C_m \square C_n$ admits a $\Gamma$-distance magic labeling.
\end{prob}


\indent Note that if a graph $G$ of order $n$ is distance magic, then it is $\zet_n$-distance magic. Moreover there are infinitely many  distance magic graphs that at the same time are group distance magic \cite{ref_AnhCicPetTep1}. Hence  Cichacz and Froncek  stated the following conjecture.
\begin{conj}[\cite{CicFro}]\label{conj:DM->GDM}
If $G$ is a distance magic graph, then $G$ is group distance magic.
\end{conj}

In the paper we  make some progress towards solution of Problem~\ref{prob:general} by proving some necessary conditions as well as some new existence results. In particular, we prove that for $n$ even  the graph $C_{n}\square C_{n}$ has a $\Gamma$-distance magic labeling for  any Abelian group $\Gamma$ of order $n^{2}$. Moreover  we show that if  $m\neq n$, then there does not exist a $(\zet_2)^{m+n}$-distance magic labeling of the Cartesian product   $C_{2^m} \square  C_{2^{n}}$. We prove a necessary and sufficient condition for $C_{m} \square  C_{n}$ with $\gcd(m,n)=1$  to be $\Gamma$-distance magic.  Observe that the Cartesian product   $C_{2^m} \square  C_{2^{n}}$ is $\zet_{2^{m+n}}$-distance magic by Theorem~\ref{thm:cart_cycle_Sylwia} but is not distance magic by Theorem~\ref{thm:cart_cycle_integers}. Therefore, this result is the first example that shows that assumptions in Conjecture~\ref{conj:DM->GDM} cannot be relaxed, that is, the statement that if a graph $G$ of order $n$ is  $\zet_n$-distance magic then it is group distance magic is not true.

\section{Sufficient conditions}\label{sec:sufficient}

Recall that the \emph{exponent} $\exp{(\Gamma)}$  of a group $\Gamma$ of order $q$ with elements $a_1,a_2,\dots,a_q$ is the smallest possible $r$ such that $ra_i=0$ for any $a_i\in \Gamma$. It is well known that in Abelian groups,  $r=\lcm(o_1,o_2,\dots,o_q)$ where $o_i$ is the order of $a_i$ for $i=1,2,\dots,q$.

It also well known that if $\Gamma$ has an even order, then there is an element $a_i$ of even order, and hence $\exp{(\Gamma)}=r$ must be even. Because the non-existence of $\Gamma$-labelings of $C_m\Box C_n$ for $|\Gamma|=mn$ odd follows from Theorem~\ref{thm:odd-m,n}, we will from now on only consider the case where  $|\Gamma|=mn$ is even. Consequently, we will always have $\exp{(\Gamma)}=r$ even.

We start with the following general theorem for Cartesian product of graphs:

\begin{thm}\label{thm:product}
Let $\Gamma_1$ and $\Gamma_2$ be Abelian groups with exponents $r_1$ and $r_2$, respectively.
Let $a_1$ and $a_2$ be some positive integers. If an $a_1r_{1}$-regular graph $G_{1}$ is $\Gamma _{1}$-distance magic and an $%
a_2r_{2}$ -regular graph $G_{2}$ is $\Gamma _{2}$-distance magic, then the
Cartesian product $G_{1}\Box G_{2}$ is $\Gamma _{1}\times \Gamma _{2}$%
-distance magic.
\end{thm}

\begin{proof} 
	Let $\ell _{i}\colon V(G_{i})\rightarrow \Gamma_{i}$ be a $\Gamma _{i}$-distance magic labeling, and $\mu _{i}$ the magic constant for the graph $G_{i}$, $i\in \{1,2\}$. Define the labeling $ \ell :V(G_{1}\Box G_{2})\rightarrow \Gamma _{1}\times \Gamma _{2}$ for $G_{1}\Box G_{2}$, as:
\begin{equation*}
\ell ((x,y))=(\ell _{1}(x),\ell _{2}(y)).
\end{equation*}%
Obviously, $\ell $ is a bijection and moreover, for any $(u,w)\in
V(G_{1}\Box G_{2})$:
\begin{eqnarray*}
w(u,w) 	&=&\sum_{(x,y)\in N_{G_{1}\Box G_{2}}((u,w))}{\ell (x,y)}\\[6pt]
			&=&\left(\sum_{x\in N_{G_{1}}(u)}\ell _{1}(x) +a_2r_{2}\ell(u),\sum_{y\in N_{G_{2}}(w)}\ell_{2}(y)+a_1r_1\ell(w)\right) \\[8pt]
			&=&(\mu _{1},\mu _{2})=\mu ,
\end{eqnarray*}%
which settles the proof.\end{proof}

Theorem~\ref{thm:product} implies the following observation.
\begin{obs}
Let $d \equiv 0 \mod 4$. A hypercube $\cQ_d$ is $\Gamma$-distance magic for any Abelian $\Gamma$ of order $2^d$ with $\exp{(\Gamma)}\leq4$.
\end{obs}
\begin{proof} 
	Note that in the factorization of $\Gamma$ we have only factors $\zet_2$ and $\zet_4$ since $\exp{(\Gamma)}\leq4$. The proof is by induction on $d$. Because $\cQ_4\cong C_4 \Box C_4$ we obtain by Theorems~\ref{thm:cart_cycle_Sylwia} and \ref{thm:C_2^n_by_Z_2} that $\cQ_4$ is $\Gamma$-distance magic for any Abelian $\Gamma$ of order $16$ with $\exp{(\Gamma)}\leq4$. Recall that for $d\geq8$ the  hypercube $\cQ_d$ can be also defined recursively in terms of the Cartesian product of two graphs as   $\cQ_d=\cQ_{d-4}\Box \cQ_4$. Obviously $\cQ_d$ is $d$-regular.
Therefore we are done by Theorem~\ref{thm:product}.
\end{proof}

Let $V(C_m \square C_n)=\{x_{i,j}:0\leq i\leq m-1, 0\leq j\leq n-1\}$, where $N(x_{i,j})=\{x_{i,j-1},x_{i,j+1},x_{i+1,j},x_{i-1,j}\}$ and the operations on the first and  second suffix are performed modulo $m$ and $n$, respectively. Without loss of generality we can assume $m\leq n$. By a diagonal $D^j$ of $C_m \square C_n$ we mean a sequence of vertices $(x_{0,j},x_{1,j+1},\ldots,$ $x_{m-1,j+m-1},x_{0,j+m}, x_{1,j+m+1},$ $\ldots,x_{m-1,j-1})$ of length $l$. It is easy to observe that $l = \lcm(m,n)$, the least common multiple of $m$ and $n$. We denote the $j$-th diagonal by $D^j = (d^j_0,d^j_1,\ldots,d^j_{l-1})$ and call $D^0$ the \textit{main diagonal}.

Now we slightly strengthen Theorem~\ref{thm:cart_cycle_Sylwia}.

\begin{thm}\label{thm:lcm/2}
Let $mn$ be even and $l =\lcm(m,n)$. Then $C_m \square C_n$ has a  $\zet_{\alpha}\times \gA$-magic labeling for any $\alpha \equiv 0 \pmod {l/2}$ and any Abelian group $\gA$ of order $mn/\alpha$.
\end{thm}
\begin{proof} 
Notice that $l=2k$ for some $k$ and $\alpha=kh$ for some $h$. Notice that $d=\gcd(m,n)$ is the number of diagonals of
$C_m \square C_n$.

%
%

For $\alpha\equiv0\pmod{l}$ the claim follows from Theorem~\ref{thm:cart_cycle_Sylwia}. Hence, we can can only look at the case when $\alpha \equiv l/2 \pmod l$.

Let $r=mn/\alpha$ and $\Gamma\cong\zet_{\alpha}\times \gA$, thus if $g \in \Gamma$, then we can write that $g=(j,a_i)$ for $j \in \zet_{\alpha}$ and $a_i \in \gA$ for $i=0,1,\ldots,r-1$. We can assume that $a_0$ is the identity in $A$. Let $\ell(x)=(l_1(x),l_2(x))$ where $l_1(x)\in\zet_{\alpha}$ and $l_2(x)\in\gA$.

There exists a subgroup $\langle h' \rangle$  of $\zet_{\alpha}$ of order $k=l/2$, therefore the element $h=(h',a_0)$ generates a subgroup $H$ in $\Gamma$ of order $k$. Let $b_0,b_1,\ldots,b_{2d-1}$ be the set of coset representatives for $\Gamma/H$. Notice that in any cyclic  group $\zet_{2j}$, $j\geq 1$ there exists an element $g\neq0$ such that there is no $a\in \zet_{2j}$ satisfying $2a=g$ (for instance take $g=1\in \zet_{2j}$). Thus by Fundamental Theorem of Abstract Algebra, because $|\Gamma/H|$ is even, we can assume without loss of generality that $b_1\in \Gamma/H$ is such that $b_1\neq 2b$ for any $b\in \Gamma/H$. Moreover we can partition $\Gamma/H$ in to $d$ pairs $(h_i,h_i')$,  where $h_i+h_i'=b_1$ and $h_i\neq h_i'$ for $i=0,1,\ldots,d-1$.

Label the vertices of $D^0$ as follows:
$$\ell(d^0_{2i})=ih+h_0,\; \ell(d^0_{2i+1})=-ih-h_0+b_1$$
for $i=0,1,\ldots,k-1$.

The vertices in $D^1,D^2,D^3,\dots,D^{d-1}$ will be labeled as
\begin{eqnarray*}
\ell(d^{j}_{g})&=l_1(d^{j-1}_g)+h_{j+1} \ \text{if} \ g\equiv  1 \imod 2,\\
\ell(d^{j}_{g})&=l_1(d^{j-1}_g)-h_{j+1} \ \text{if} \ g\equiv  0 \imod 2.
\end{eqnarray*}

Observe that the labeling $\ell$ is a bijection because $h_j\neq -h_i+b_1$ for any $i\neq j$. Moreover,
\begin{eqnarray*}
	\ell(d^{j}_{2i})+\ell(d^{j}_{2i+1}) 	&= &b_1 \hskip35pt \text{and}\\
	\ell(d^{j}_{2i+1})+\ell(d^{j}_{2i+2}) &= & h+b_1
\end{eqnarray*}
for any $i$.

If $d>2$, then the vertex $x_{i',j'}=d^j_i$ has in $C_m \square C_n$ neighbors $d^{j-1}_{i},d^{j-1}_{i+1},d^{j+1}_{i-1}$ and $d^{j+1}_{i}$. Therefore $w(d^j_i)=h+2b_1$ and the labeling is $\Gamma$-distance magic as desired.

If $d\leq2$, then the vertex $x_{i',j'}$ has in $C_m \square C_n$ neighbors $d^{j}_{a},d^{j}_{a+1},d^{j}_{b-1}$ and $d^{j}_{b} $ for $0\leq j\leq 1$, $0\leq a<b\leq l-1$. We know that at least one of $m,n$ is even, so
we can assume that $m = 2s$. Because $d^{j}_{a}=x_{i',j'-1}$  and $d^{j}_{b}= x_{i',j'+1}$, it is clear that
$a = b + qm$ for same $1\leq q  < l/m$. But $m = 2s$ and $a = b + 2qs$ and hence $a$ and $b$
have the same parity. When $a$ and $b$ are even, say $a = 2c$ and $b = 2f$, then
$$d^{j}_{a}+d^{j}_{a+1}=d^{j}_{2s}+d^{j}_{2s+1}=b_1$$
and
$$d^{j}_{b-1}+d^{j}_{b}=d^{j}_{2f-1}+d^{j}_{2f}=b_1+h,$$
which implies $w(x_{i',j'})=h+2b_1$.\end{proof}

Now we present a class of group distance magic cycle products, that is, cycle products that are $\Gamma$-distance magic for any Abelian group $\Gamma$ of an appropriate order.

\begin{thm}\label{thm:C_n x C_n}
Let $n$ be even. Then $C_{n} \square C_{n}$ has a  $\Gamma$-distance magic labeling for any  Abelian group $\Gamma$ of order $n^2$.
\end{thm}
\begin{proof} 
The Fundamental Theorem of Finite Abelian Groups states that a finite Abelian
group $\Gamma$ of order $m=n^2$ can be expressed as the direct product of cyclic subgroups of prime-power order. This implies that
$$
	\Gamma\cong\zet_{p_1^{\alpha_1}}\times\zet_{p_2^{\alpha_2}}\times\ldots\times\zet_{p_k^{\alpha_k}}\;\;\; \mathrm{where}\;\;\; n = p_1^{\alpha_1}\cdot p_2^{\alpha_2}\cdot\ldots\cdot p_k^{\alpha_k}
$$
and $p_i$ for $i \in \{1, 2,\ldots,k\}$ are primes, not necessarily distinct. This product is unique up to the order of the direct product.
Therefore there exists $H<\Gamma$ such that  $|H|=n$.  Let $b_0,b_1,\ldots,b_{n-1}$ be coset representatives of $\Gamma/H$. Recall that in any Abelian group of even order the number of involutions is odd, therefore 
 $\Gamma/H$ has $2t-1$ involutions $\iota_1,\iota_2,\ldots,\iota_{2t-1}$ for $t\geq 1$.


 Let $b_0$ be the identity element of $\Gamma/H$ and $b_i=\iota_i$ for $i\in\{1,2,\ldots,2t-1\}$, and $b_{i+1} = -b_i$ for $i \in\{ 2t,2t+2,2t+4,\ldots,n-2\}$.

Observe that because  $|H|$ is even there exists an involution  $\iota\in H$ ($\iota\neq 0$ and $2\iota=0$). We will define a bijection  $\varphi\colon H \rightarrow H$ such that $\varphi(x)\neq x$ for any $x\in H$.  For $n\equiv 0 \pmod 4$ let   $\varphi(x)=x+\iota$. When $n\equiv 2 \pmod 4$, then let $\varphi(x)=-x+\iota$.  Notice that for $n\equiv 2 \pmod 4$   there does not exist $x\in H$ such that $2x=\iota$.

Hence we   can partition $H$ in to $n/2$ pairs $(h_i,h_i')$  such that $h_i'=\varphi(h_i)$ and $h_i\neq h_i'$ for $i=0,1,\ldots,n/2-1$.

Now, for $i=0,1,\dots,n-1$, label the vertices of $D^0$ as
$$\ell(d^0_{i})=h_i,\;\;\;\ell(d^0_{n/2+i})=\varphi(h_i)
$$
and the vertices of $D^2$ as
\begin{eqnarray*}
	\ell(d^2_{i})	&=-\ell(d^0_{i+1})+b_2 & \text{if} \ i\equiv  0 \imod 2 \hskip10pt \text{and}\\
	\ell(d^2_{i})	&=-\ell(d^0_{i+1})-b_2 &\text{if} \ i\equiv  1 \imod 2.
\end{eqnarray*}
For $r\in\{1,3\}$ label  the vertices of $D^r$ as
\begin{eqnarray*}
	\ell(d^r_{i})	&=\ell(d^{r-1}_{i})-b_{r-1}+b_r & \text{if} \ i\equiv  0 \imod 2 \hskip10pt \text{and}\\
	\ell(d^r_{i})	&=\ell(d^{r-1}_{i})+b_{r-1}-b_r &\text{if} \ i\equiv  1 \imod 2.
\end{eqnarray*}
%
%
%
For $r\in\{4,5,\ldots,n-1\}$ the vertices in $D^r$ will be labeled as
\begin{eqnarray*}
	\ell(d^{r}_{i})	&=\ell(d^{r-4}_{i+2})-b_{r-4}+b_{r} & \text{if} \  i\equiv  0 \imod 2\hskip10pt \text{and}\\
	\ell(d^{r}_{i})	&=\ell(d^{r-4}_{i+2})+b_{r-4}-b_{r} & \text{if} \  i\equiv  1 \imod 2.
\end{eqnarray*}
Observe that
$$w(d^{r}_{i})= \ell(d^{r-1}_{i})+\ell(d^{r-1}_{i+1})+\ell(d^{r+1}_{i-1})+\ell(d^{r+1}_{i}).$$

Assume first that $r\not \in \{n-1,0\}$. When $r\equiv 1,2\pmod 4$, then
$$
	w(d^{r}_{i})= \ell(d^{r-1}_{i})+\ell(d^{r-1}_{i+1})-\ell(d^{r-1}_{i})-\ell(d^{r-1}_{i+1})+2\iota=0.
$$
If
$r\equiv 0,3\pmod 4$,
then
$$
	w(d^{r}_{i})= -\ell(d^{r-3}_{i+1})-\ell(d^{r-3}_{i+2})+2\iota+\ell(d^{r-3}_{i+1})+\ell(d^{r-3}_{i+2})=0.
$$

Assume now that $r=n-1$ and $n\equiv 0 \pmod 4$. Then
\begin{eqnarray*}
	w(d^{n-1}_{i})
	&=&\ell(d^{n-2}_{i})+\ell(d^{n-2}_{i+1})+\ell(d^{0}_{i-1})+\ell(d^{0}_{i})\\
	&=&\ell(d^{2}_{i+2(n/4-1)})+\ell(d^{2}_{i+1+2(n/4-1)})+\ell(d^{0}_{i-1})+\ell(d^{0}_{i})\\
	&=&\ell(d^2_{i-2+n/2})+\ell(d^2_{i-1+n/2})+\ell(d^{0}_{i-1})+\ell(d^{0}_{i})\\
	&=&-\ell(d^0_{i-1+n/2})-\ell(d^0_{i+n/2})+\ell(d^{0}_{i-1})+\ell(d^{0}_{i})\\
	&=&-h_{i-1}-h_{i}+h_{i-1}+h_i=0.
\end{eqnarray*}
If $r=n-1$ and $n\equiv 2 \pmod 4$, then
\begin{eqnarray*}
	w(d^{n-1}_{i})
		&=&\ell(d^{n-2}_{i})+\ell(d^{n-2}_{i+1})+\ell(d^{0}_{i-1})+\ell(d^{0}_{i})\\
		&=&\ell(d^{0}_{i+(n/2-1)})+\ell(d^{0}_{i+1+(n/2-1)})+\ell(d^{0}_{i-1})+\ell(d^{0}_{i})\\
		&=&\ell(d^0_{i-1+n/2})+\ell(d^0_{i+n/2})+\ell(d^{0}_{i-1})+\ell(d^{0}_{i})\\
		&=&h'_{i-1}+h'_{i}+h_{i-1}+h_i\\
		&=&0.
\end{eqnarray*}

Similarly we obtain $w(d^{0}_{i})=0$. Hence the labeling is $\Gamma$-distance magic as desired.\end{proof}

Theorem~\ref{thm:C_n x C_n} now immediately implies the following.
\begin{cor}\label{cor:C_2^n}
	The graph $C_{2^n}\square C_{2^n}$ has a $\Gamma$-distance magic labeling for $n \geq 2$ and any Abelian group $\Gamma$ of order $2^{2n}$.
\end{cor}

\section{Necessary  conditions}\label{sec:necessary}

Now we present theorems showing that if we have a group $\Gamma\cong\zet_{p_1^{\alpha_1}}\times\zet_{p_2^{\alpha_2}}\times\ldots\times\zet_{p_k^{\alpha_k}}$ with elements $(g_1,g_2,\dots,g_k)$  and the exponent $r=\exp(\Gamma)$ is rather small compared with the length of the diagonal of $C_{m} \square  C_{n}$, then there is no $\Gamma$-distance magic labeling of the cycle product. In other words, the results are showing that if
some entries $g_i$ of $(g_1,g_2,\dots,g_k)\in\Gamma$ would have to repeat too many times, then such labeling does not exist.

For a positive integer $m$  define a function

$$f(m)=\left\{
\begin{array}{lll}
m/4&\textrm{if}&m\equiv0 \pmod 4,\\
m/2 &\textrm{if}&m\equiv2 \pmod 4,\\
m &\textrm{if}&m\equiv1 \pmod 2.
\end{array}\right.$$

\begin{thm}\label{thm:main-f} 
	Let $\Gamma$ be an Abelian group of an even order $mn$ with exponent $r$. If $2r\min\{f(m),f(n)\} <\lcm(m,n)$, then there does not exist a $\Gamma$-distance magic labeling of the Cartesian product   $C_{m} \square  C_{n}$.
\end{thm}

\begin{proof}
 For the sake of contradiction, suppose that there exists a  $\Gamma$-distance magic labeling $\ell$ of the Cartesian product   $C_{m} \square  C_{n}$ with magic constant $\mu$.   By our assumption, we have $mn$  even. Without loss of generality we can assume that $m<n$ and  $\ell(x_{0,0})=0$.
 Let us consider the weights of
$x_{0,1}$ and $x_{m-1,2}$:
$$
   w(x_{0,1})=\ell(x_{0,0})+\ell(x_{1,1})+\ell(x_{m-1,1})+\ell(x_{0,2})
$$
and
$$
   w(x_{m-1,2})=\ell(x_{m-1,1})+\ell(x_{0,2})+\ell(x_{m-2,2})+\ell(x_{m-1,3}).
$$
Because we assumed that $\ell$ is a  $\Gamma$-distance magic labeling, we have $w(x_{0,1})=w(x_{m-1,2})=\mu$, which yields
$$
\ell(x_{0,0})+\ell(x_{1,1})+\ell(x_{m-1,1})+\ell(x_{0,2})=\ell(x_{m-1,1})+\ell(x_{0,2})+\ell(x_{m-2,2})+\ell(x_{m-1,3}).
$$
and hence
$$
\ell(x_{0,0})+\ell(x_{1,1})=\ell(x_{m-2,2})+\ell(x_{m-1,3}).
$$
Similarly, comparing weights of vertices $x_{m-2,3}$ and $x_{m-3,4}$ we obtain
\begin{align*}
   \ell(x_{m-2,2})+\ell(x_{m-1,3})+\ell(x_{m-3,3})+\ell(x_{m-2,4})=\\
   \ell(x_{m-3,3})+\ell(x_{m-2,4})+\ell(x_{{m}-4,4})+\ell(x_{{m}-3,5})
\end{align*}
and thus
$$
\ell(x_{m-2,2})+\ell(x_{m-1,3})=\ell(x_{{m}-4,4})+\ell(x_{{m}-3,5}).
$$
which implies
$$
\ell(x_{0,0})+\ell(x_{1,1})=\ell(x_{{m}-4,4})+\ell(x_{{m}-3,5}).
$$
Repeating that procedure we conclude that
$$
   \ell(x_{0,0})+\ell(x_{1,1})=\ell(x_{-2\alpha,2\alpha})+\ell(x_{1-2\alpha,1+2\alpha})=a_0
$$
for some $a_0\in \Gamma$ and any natural number $\alpha$.

Recall that by the {main diagonal} of $C_m \square C_n$ we mean the cyclic sequence of vertices $(x_{0,0},x_{1,1},\ldots,$ $x_{m-1,m-1},x_{0,m}, x_{1,m+1},$ $\ldots,x_{m-1,n-1})$ of length $l= \lcm(m,n)$.  We now consider the following system of equations, going along the main diagonal. Notice that the subscripts need to be read modulo $m$ and $n$, respectively. So, for instance, when $m<n$, the vertex denoted by $x_{m,m}$ is in fact $x_{0,m}$.

Analogously, we get
\begin{align*}
\ell(x_{j,j})+\ell(x_{j+1,j+1})=\ell(x_{j-2\alpha,j+2\alpha})+\ell(x_{j+1-2\alpha,j+1+2\alpha})=a_j
\end{align*}
for some $a_j\in \Gamma$ and any natural number $\alpha$.

Note that $x_{-2\alpha,2\alpha}$ belongs to the same diagonal as $x_{0,0}$ for $2\alpha\equiv-2\alpha\pmod m$ (then $x_{-2\alpha,2\alpha}=x_{2\alpha,2\alpha}$) or $2\alpha\equiv-2\alpha \pmod n$ (then $x_{-2\alpha,2\alpha}=x_{-2\alpha,-2\alpha}$), what happens for both $\alpha=f(m)$ and $\alpha=f(n)$.

This implies that taking $k=\min\{f(m),f(n)\}$ we obtain:
\begin{align}\label{eq:rownanie}
\ell(x_{j,j})+\ell(x_{j+1,j+1})=\ell(x_{j+2k,j+2k})+\ell(x_{j+1+2k,j+1+2k}).
\end{align}

 Note that $x_{2rk,2rk}\neq x_{0,0}$ since $2rk< \lcm(m,n)$.   Also, the elements $a_0,a_1,\dots,a_{2k-1}$ are not necessarily all distinct.

\begin{align}\label{uklad}
\left\{
\begin{array}{lll}
  \ell(x_{0,0})&+\ell(x_{1,1}) & =a_0 \\
  \ell(x_{1,1})&+\ell(x_{2,2}) & =a_1 \\
     \vdots\\
  \ell(x_{2k-1,2k-1})&+\ell(x_{2k,2k})&=a_{2k-1} \\
  \ell(x_{2k,2k})&+\ell(x_{2k+1,2k+1}) & =a_0 \\
  \ell(x_{2k+1,2k+1})&+\ell(x_{2k+2,2k+2}) & =a_1 \\
     \vdots\\
  \ell(x_{4k-1,4k-1})&+\ell(x_{4k,4k})&=a_{2k-1} \\
	 \vdots\\
  \ell(x_{2(r-1)k,2(r-1)k})&+\ell(x_{2(r-1)k+1,2(r-1)k+1}) & =a_0 \\
  \ell(x_{2(r-1)k+1,(r-1)k+1})&+\ell(x_{2(r-1)k+2,2(r-1)k+2}) & =a_1 \\
     \vdots\\
  \ell(x_{2rk-1,2rk-1})&+\ell(x_{2rk,2rk})&=a_{2k-1} .\\
\end{array}\right.
\end{align}
Multiplying every other equation by $-1$, starting with
$$-\ell(x_{1,1})-\ell(x_{2,2})  =-a_1,$$
and adding all equations, we obtain
$$
\ell(x_{0,0})-\ell(x_{2rk,2rk})  =r\sum_{i=0}^{2k-1}(-1)^{i}a_i.
$$
Recall that $\ell(x_{0,0})=0$. Because $r=\exp(\Gamma)$, we have $r\sum_{i=0}^{2k-1}(-1)^{i}a_i=0$. This implies
$-\ell(x_{2rk,2rk})=0$, which is a contradiction, because the labeling is injective and we have assumed that  $x_{2rk,2rk}\neq x_{0,0}$.
\end{proof}

The following theorem gives a similar result in terms of a more obvious bound, using $\gcd(m,n)$, the number of diagonals in $C_{m} \square  C_{n}$.

\begin{thm}\label{thm:gcd} Let $\Gamma$ be an Abelian group of an even order $mn$ with exponent $r$. If $2r\gcd(m,n) <\lcm(m,n)$, then there does not exist a $\Gamma$-distance magic labeling of the Cartesian product   $C_{m} \square  C_{n}$.
\end{thm}

\begin{proof}
We again use  contradiction and assume that there exists a  $\Gamma$-distance magic labeling $\ell$ of   $C_{m} \square  C_{n}$ with magic constant $\mu$ and $mn$ is even, $m<n$ and  $\ell(x_{0,0})=0$.

By the \emph{first backward diagonal} of $C_m \square C_n$ we mean the cyclic sequence of vertices
$(x_{0,1},x_{1,0},x_{2,n-1}\ldots,$ $x_{m-1,n-m+2},x_{0,n-m+1},x_{1,n-m},\ldots,x_{m-1,2})$
of length $l = \lcm(m,n)$.
Similarly, the  sequence
$(x_{0,2},x_{1,1},x_{2,0}\ldots,x_{m-1,n-m+3},$
\linebreak
$
x_{0,n-m+2},
x_{1,n-m+1},\ldots,x_{m-1,3})$ is the  \emph{second backward diagonal} and so on.

Set $k=\gcd(m,n)$. Because the length of each backward diagonal is $\lcm(m,n)$, there are $k$ backward diagonals, and the vertices
$
x_{0,k+1},x_{1,k},x_{2,k-1}$ $\ldots,
x_{0,2k+1},x_{1,2k},x_{2,2k-1}\ldots,
$ $x_{k,k+1},\ldots, x_{0,3k+1},x_{1,3k},x_{2,3k-1},\ldots,x_{2k,2k+1}, \dots
$
belong to the same diagonal as $x_{0,1}$.

We look at weights of vertices of that sequence, starting at $x_{0,1}$ and going the opposite direction, that is, $w(x_{0,1}),w(x_{m-1,2}),w(x_{m-2,3}),\ldots,w(x_{1,0})$.

The weights of
$x_{0,1}$ and $x_{m-1,2}$ are again
$$
w(x_{0,1})=\ell(x_{0,0})+\ell(x_{1,1})+\ell(x_{m-1,1})+\ell(x_{0,2})=\mu
$$
and
$$
w(x_{m-1,2})=\ell(x_{m-1,1})+\ell(x_{0,2})+\ell(x_{m-2,2})+\ell(x_{m-1,3})=\mu.
$$
Comparing the above equalities, we get
$$
\ell(x_{0,0})+\ell(x_{1,1})=\ell(x_{m-2,2})+\ell(x_{m-1,3}).
$$
Continuing this way, we obtain
\begin{align}\label{eq:rownanie-gcd-a_0}
\ell(x_{0,0})+\ell(x_{1,1})=\ell(x_{-2\alpha_0,2\alpha_0})+\ell(x_{1-2\alpha_0,1+2\alpha_0})=c_1
\end{align}
for some $c_1\in \Gamma$ and any natural number $\alpha_0$. In particular, because there are $k$  diagonals and $x_{2k,2k+1}$ belongs to the same backward diagonal as $x_{0,1}$, we must also have\\
%

%
$$\ell(x_{0,0})+\ell(x_{1,1})=\ell(x_{2k,2k})+\ell(x_{2k+1,2k+1})=c_1,$$
$$\ell(x_{0,0})+\ell(x_{1,1})=\ell(x_{4k,4k})+\ell(x_{4k+1,4k+1})=c_1,$$
$$\vdots$$
$$\ell(x_{0,0})+\ell(x_{1,1})=\ell(x_{2(r-1)k,2(r-1)k})+\ell(x_{2(r-1)k+1,2(r-1)k+1})=c_1.$$
%
Looking at the first backward diagonal again and starting to compare the weights at $x_{1,0},x_{0,1},x_{m-,2},\dots$ instead, we get
$$
w(x_{1,0})=\ell(x_{1,n-1})+\ell(x_{2,0})+\ell(x_{0,0})+\ell(x_{1,1})=\mu
$$
and
$$
w(x_{0,1})=\ell(x_{0,0})+\ell(x_{1,1})+\ell(x_{m-1,1})+\ell(x_{0,2})=\mu.
$$
 Comparing these equalities, we obtain
$$
\ell(x_{1,n-1})+\ell(x_{2,0})=\ell(x_{m-1,1})+\ell(x_{0,2})=d_1
$$
for some element $d_1$.

Comparing the weights of remaining vertices, we again have
$$
\ell(x_{1,n-1})+\ell(x_{2,0})=\ell(x_{1-2\beta_1,-1+2\beta_1})+\ell(x_{2-2\beta_1,2\beta_1})=d_1
$$
for any $\beta_1\in\zet$.

Proceeding in the same fashion, we obtain two such equalities for each diagonal. Ordering them conveniently and renaming the elements $c_i$ and $d_i$, we again obtain the same system of equations as in the previous theorem. Namely,

$$
\left\{
\begin{array}{lll}
\ell(x_{0,0})&+\ell(x_{1,1}) & =a_0 \\
\ell(x_{1,1})&+\ell(x_{2,2}) & =a_1 \\
\vdots\\
\ell(x_{2k-1,2k-1})&+\ell(x_{2k,2k})&=a_{2k-1} \\
\ell(x_{2k,2k})&+\ell(x_{2k+1,2k+1}) & =a_0 \\
\ell(x_{2k+1,2k+1})&+\ell(x_{2k+2,2k+2}) & =a_1 \\
\vdots\\
\ell(x_{4k-1,4k-1})&+\ell(x_{4k,4k})&=a_{2k-1} \\
\vdots\\
\ell(x_{2(r-1)k,2(r-1)k})&+\ell(x_{2(r-1)k+1,2(r-1)k+1}) & =a_0 \\
\ell(x_{2(r-1)k+1,(r-1)k+1})&+\ell(x_{2(r-1)k+2,2(r-1)k+2}) & =a_1 \\
\vdots\\
\ell(x_{2rk-1,2rk-1})&+\ell(x_{2rk,2rk})&=a_{2k-1} .\\
\end{array}\right.
$$
Solving it the same way as before, we again obtain
$$\ell(x_{2rk,2rk})=r\sum_{i=0}^{2k-1}(-1)^{i}a_i=0$$
by the property of group exponent $r$. But because $\ell(x_{0,0})=0$ and $x_{2rk,2rk}\neq x_{0,0}$, we get a contradiction, which completes the proof.
\end{proof}

	
	Notice that each of the above two theorems is useful in different scenarios. For instance, when $m\equiv0\pmod4$ and $n=mq$, then Theorem~\ref{thm:main-f} gives a stronger result; when $m$ is odd and $n\neq mq$, then it is better to use Theorem~\ref{thm:gcd}.

To illustrate the strength of the theorems above with a concrete example, we present two special cases separately. 

\begin{cor}\label{cor:many-Z_s}
	Let $\Gamma\cong (\zet_s)^t$ when $s$ is even or $\Gamma\cong (\zet_s)^t\times\zet_2$ when $s$ is odd, and $|\Gamma|=mn$. Then $C_m\square C_n$ does not have a $\Gamma$-distance magic labeling if $n>sm$.
\end{cor}

\begin{proof}
	Follows directly from Theorem~\ref{thm:gcd}. 
\end{proof}

The following non-existence result  is in a certain sense `dual' to the statement above. We show that for a given cycle length $m$, we can always find a corresponding long cycle $C_n$ and an Abelian group $\Gamma$ which does not provide a labeling of $C_m\square C_n$.

\begin{obs}\label{obs:long-C_n-m-only}
	For any positive integer $m$ there exists $n$ for which the Cartesian product $C_m\square C_n$ is not group distance magic. That is, there exists an Abelian group $\Gamma$ such that $C_m\square C_n$ is not $\Gamma$-distance magic.
\end{obs}

\begin{proof}
	Let $n=2m^3$ and $\Gamma\equiv(\zet_m)^4\times\zet_2$. Then $\exp(\Gamma)\leq2m$ and $\gcd(m,n)=m$. Hence we have
	$2r\gcd(m,n)\leq 4m^2<2m^3$, because $m\geq3$. But $\lcm(m,n)=n=2m^3$, and the product $C_m\square C_n$ is not $\Gamma$-distance magic by Theorem~\ref{thm:gcd}.
\end{proof}

The following extreme case is worth mentioning.

\begin{obs}\label{gdc} 
Let  $\gcd(m,n)=1$. There  exists a $\Gamma$-distance magic labeling of the Cartesian product   $C_{m} \square  C_{n}$ if and only if $mn$ is even and either $\Gamma\cong \zet_2\times \zet_{mn/2}$ or
$\Gamma\cong \zet_{mn}$.
\end{obs}

\begin{proof}
	For $\Gamma\cong \zet_{mn}$ and $\Gamma\cong \zet_2\times \zet_{mn/2}$ the labelings exist by Theorem~\ref{conj:DM->GDM}. 
	
	Because $\gcd(m,n)=1$, we cannot have both $m$ and $n$ even, so without loss of generality $m$ is odd.  If $n$ is odd, then $C_m \square C_n$ is not $\Gamma$-distance magic graph for  any Abelian group $\Gamma$ of order $mn$ by Theorem~\ref{thm:odd-m,n}. 
	
It is well known that for an Abelian group $\Gamma$ of order $2k$, $\exp(\Gamma) =k$ if and only if $k$ is even and $\Gamma\cong\zet_2\times\zet_k$. 
	
	Therefore, if  $\Gamma\ncong \zet_{mn}$ and $\Gamma\ncong \zet_2\times \zet_{mn/2}$, we must have  $\exp(\Gamma)< mn/2$. On the other hand, we have $\lcm(m,n)=mn$, because $\gcd(m,n)=1$. Hence 
	$2r\gcd(m,n)<mn=\lcm(m,n)$ and 
	there does not exist a  $\Gamma$-distance magic labeling of   $C_{m} \square  C_{n}$ by Theorem~\ref{thm:gcd}.
\end{proof}

Finally, we complement Theorem~\ref{thm:C_2^n_by_Z_2} and Observation~\ref{obs:C_2^n} by the following related result.

\begin{thm}\label{obs:C_2^n}
	There  exists a $(\zet_2)^{m+n}$-distance magic labeling of the Cartesian product   $C_{2^m} \square  C_{2^{n}}$ if and only if $2\leq m =n$.
\end{thm}

\begin{proof}
	When $m=n$, then we are done by Theorem~\ref{thm:C_2^n_by_Z_2}. Suppose that $m\neq n$. Without loss of generality we can assume that $2\leq m<n$; then $\min\{f(2^m),f(2^n),\gcd(2^m,2^n)\}=2^{m-2}$, whereas $\lcm(2^m,2^n)=2^n$. Observe that the exponent of $(\zet_2)^{m+n}$ is $r=2$. Therefore $2r\min\{f(m),f(n)\}=2^{m} <2^{n}=\lcm(m,n)$.
\end{proof}

\section{Conclusion}

We made some progress towards the full characterization of Abelian groups $\Gamma$ such that $|\Gamma|=mn$ and there exists  a $\Gamma$-distance magic labeling of $C_m\square C_n$. 

We improved a previous bound for existence of such labeling, showing that if $\Gamma$ has a cyclic subgroup of order $\lcm(m,n)/2$, then $C_m\square C_n$ is $\Gamma$-distance magic, lowering the bound from $\lcm(m,n)$.

On the other hand, we have shown that groups with an exponent $\exp(\Gamma)$ that is relatively small compared with $\lcm(m,n)$ do not admit such labeling.

Since we found necessary conditions for existence of such labeling (Theorem~\ref{thm:main-f} and~\ref{thm:gcd}), which in some cases were also sufficient (Observation~\ref{gdc}), we post now the following conjecture.

\begin{conj}
	Let $\Gamma$ be an Abelian group of an even order $mn$ with exponent $r$.  There exists a $\Gamma$-distance magic labeling of the Cartesian product   $C_{m} \square  C_{n}$ if and only if $2r\min\{f(m),f(n),\gcd(m,n)\} \ge\lcm(m,n)$.
\end{conj}

\bibliographystyle{plain}

\end{document}